\documentclass{no_journal}
\usepackage{bm}
\usepackage{caption}
\usepackage{subcaption}
\usepackage{xcolor}
\usepackage{dsfont}
\usepackage[square,numbers]{natbib} 
\usepackage{comment}
\usepackage{graphicx,natbib,latexsym}
\usepackage{dsfont}
\usepackage{multirow}
\usepackage{float}
\usepackage{rotating}
\usepackage{moresize}
\usepackage{enumerate}
\usepackage{tabularx}
\usepackage{color} 
\usepackage{amsmath,amsfonts,amssymb}
\usepackage{graphicx} 
\usepackage{mathtools}
\usepackage{dsfont}
\usepackage{float}
\usepackage{xcolor}
\usepackage{subcaption}
\usepackage{hyperref}
\hypersetup{colorlinks=true, linkcolor=blue, urlcolor=blue, citecolor=blue}

\theoremstyle{plain}
\newtheorem{theo}{Theorem}
\newtheorem{lemm}{Lemma}
\newtheorem{coro}{Corollary}
\newtheorem{defi}{Definition}
\newtheorem{assu}{Assumption}

\theoremstyle{remark}
\newtheorem{exam}{Example}
\newtheorem{rema}{Remark}

\renewcommand{\P}{{\mathbb P}}
\renewcommand{\d}{{\mathrm d}}
\renewcommand{\hat}{\widehat}
\renewcommand{\tilde}{\widetilde}
\newcommand{\E}{{\mathbb E}}
\newcommand{\Reals}{\mathbb{R}} 
\newcommand{\R}{\mathbb{R}}
\newcommand{\N}{\mathbb N}

\newcommand{\lebesgue}{\mathcal{L}}
\newcommand{\I}[1]{\mathds{1}_{\{#1\}}}

\newcommand{\dist}{\mathrm{dist}}

\newcommand{\norm}[1]{  \| #1   \|}

\definecolor{R}{RGB}{255, 150, 0}
\definecolor{T}{RGB}{0, 100, 255}

\authornames{R.~COTSAKIS {\it et al.}} 
\shorttitle{On the spatial extent of extreme threshold exceedances} 

\begin{document}

\title{On the spatial extent of extreme threshold exceedances} 

\authorone[University of Lausanne]{Ryan Cotsakis} 
\authorone[Universit\'{e} C\^{o}te d'Azur]{Elena Di Bernardino} 
\authorone[INRAE]{Thomas Opitz} 
 
\addressone{Expertise Center for Climate Extremes (ECCE), Faculty of Business and Economics (HEC) - Faculty of Geosciences and Environment, University of Lausanne, CH-1015 Lausanne, Switzerland} 
\emailone{ryan.cotsakis@unil.ch} 

\addressone{Laboratoire J.A. Dieudonn\'{e}, UMR CNRS  7351,  Nice, 06108, France}
\emailone{Elena.Di\_bernardino@unice.fr}
\addressone{Biostatistics and Spatial Processes,  228 route de l'A\'erodrome, Avignon, 84914, France}
\emailone{thomas.opitz@inrae.fr}
 
\begin{abstract}
We introduce the \emph{extremal range}, a local statistic for studying the spatial extent of extreme events in random fields on $\Reals^{d}$. Conditioned on  exceedance of a high threshold at a location $s$, the extremal range at $s$ is the random variable defined as the smallest distance from $s  {\in \Reals^d}$ to a location where there is a non-exceedance. We leverage tools from excursion-set theory, such as Lipschitz-Killing curvatures, to express distributional properties of the extremal range, including asymptotics for small distances and high thresholds. 
The extremal range captures the rate at which the spatial extent of conditional extreme events scales for increasingly high thresholds, and we relate its distributional properties with the well-known  bivariate tail dependence coefficient and the extremal index of time series in Extreme-Value Theory. We calculate theoretical extremal-range properties for commonly used models, such as Gaussian or regularly varying random fields. 
Numerical studies illustrate that, when the extremal range is estimated from discretized excursion sets observed on compact observation windows, the distribution of the resulting estimators appropriately reproduces the theoretically derived links with the Lipschitz-Killing curvature densities.
\end{abstract}



\acks This work has been supported by the French government, through the 3IA C\^{o}te d'Azur Investments in the Future project managed by the National Research Agency (ANR) with the reference number ANR-19-P3IA-0002 and the project  France 2030  ANR-22-EXIR-0008.  Contributions to the Mathematics Stack Exchange by Willie W.Y. Wong were highly influential in constructing the proof of Lemma~\ref{lem:continuity}.


\competing 
There were no competing interests to declare which arose during the preparation or publication process of this article.

\data 
Code related to the simulations found in Section \ref{numericSection}  can be found at \url{https://github.com/napped-eel-pecan/Extremal-Range}.  
  
 \section{Introduction}
{The spatial  dependence of extreme events is an extensively studied topic in the theory of stochastic processes in $\mathbb{R}^d$ \citep[see, e.g.,][]{Leadbetter1988,Hult2005}  and is of crucial importance for modeling climatic and environmental extreme events. In applications,  extreme risks often arise from concurrence and compounding of extremes in time $(d=1)$, in geographic space  $(d=2)$ or in space and time $(d=3)$; \citep[see, e.g.,][]{dombry2018,aghakouchak2020climate}). Threshold exceedances, and the excursion sets describing the area of space where exceedance takes place, are a useful tool for theoretical and practical analyses. Using excursion sets, we here focus on studying the spatial contiguity of extreme events locally around a given location to allow for the theoretical and empirical description of the spatial extent of extreme clusters.}


Analysis of excursion sets has become valuable in spatial statistics (see, e.g., \cite{bolin2015,sommerfeld2018}) and computer vision (see, e.g.,  
\cite{bleau2000,sezgin2004}), especially for data on regular grids, such as climate model output, remote sensing data or medical images. 
We use the framework of Extreme-Value Theory \citep[EVT,][]{dehaan2006}, useful to formulate general tail-regularity assumptions.
The standard asymptotic models in spatial EVT  exhibit asymptotic dependence where the limiting dependence structure of threshold exceedances is characterized by Peaks-Over-Threshold stability \citep{ferreira2014,dombry2015,thibaud2015}. However, strong empirical evidence from many environmental processes advises against this property \citep{tawn2018,huser2022}. Often, spatial dependence between threshold exceedances is lost as thresholds are increased, and it may ultimately vanish in the case of asymptotic independence. More flexible subasymptotic models have been proposed  to accommodate asymptotic independence or even both situations of asymptotic (in)dependence (see, e.g.,  \cite{huser2017,huser2022,zhang2022}).


Here, we use a setting borrowing from the idea of studying the process conditional to exceedance above a high threshold at a location of interest to better understand spatial joint tail decay behavior near this location. This approach provides useful  representations of classical extreme-value limits such as max-stable processes \citep[e.g.,][]{Engelke2014,Bienvenue2017} and has been generalized to more flexible and statistically tractable representations of spatial extremal dependence in the conditional extremes framework \citep{heffernan2004,wadsworth2022}. 
The tail dependence coefficient $\lim_{u\rightarrow 1} \P(F_2(X_2)>u\mid F_1(X_1)>u)$ of two random variables $X_i$ following distribution functions $F_i$,  $i=1,2$, is a conditional probability that is a routinely used exploratory and diagnostic tool to assess the strength of bivariate extremal dependence \citep{coles1999}.  As noted by \cite{wadsworth2022,Huser2024}, it is common in environmental data for the spatial dependence to weaken as the considered threshold increases. One interpretation of this phenomenon, the inspiration for the statistics introduced in this paper, is that the \textit{spatial extent} of extreme events tends to decrease with an increase in the threshold level.
Thus, in this paper, we focus on the size and other geometric properties of excursion sets of continuous  random fields---the regions where the random fields exhibit threshold exceedances.

There is a vast literature concerning the geometric features of excursion sets of random fields; see \cite{adler2007} for a comprehensive introduction. Links to Extreme-Value Theory arise when behavior at extreme thresholds, or behavior of maxima, is studied \citep[e.g.,][]{Sun1993,Schlather2003,Last2023}. For smooth random fields, geometric summaries of excursion sets, namely their Lipschitz-Killing  curvatures (LKCs), carry pertinent information about the asymptotic dependence structure at extreme thresholds (see, e.g., \citep{Adler_Samorodnitsky_Taylor_2010, dibernardino2022}). 
In this paper, we introduce a new local statistic, the \textit{extremal range}. The extremal range at a site $s\in\Reals^d$ is defined as the largest radius $r$ around $s$ such that all locations within $r$ are extreme, conditioned on a threshold exceedance at~$s$. We will explore how the extremal range relates to the $d$ and $(d-1)$-dimensional Lipschitz-Killing  curvatures  of the excursion set and to the notion of asymptotic dependence defined by a positive value of the tail dependence coefficient. {Conceptually, the notion of extremal range bears some similarity to the inradius of cells in a tessellation of $\mathbb{R}^d$ (intersected with an observation window), i.e., of the radius of the largest possible hyperball that can be inscribed into a cell \citep{Chenavier2016}.}

The extremal range can be seen as a spatial analogue to the extremal index \citep{Moloney2019}, a popular asymptotic statistic for time series extremes that allows for interpretation as the reciprocal of the average number of consecutive time steps over which an extreme cluster spans. In this sense, both quantities provide a notion of the size of clusters of extremes. However, several notable distinctions can be made. Firstly, we consider the contiguous $d$-dimensional Euclidean space and not one single time dimension with regular discrete time steps. In one dimension, the distributional properties of the extremal range and its asymptotics at high thresholds can be obtained by studying sojourn times of one dimensional stochastic processes \citep{berman1971, berman1982, kratz2006, pham2013, dalmao2019}. Where the classical extremal index is equal to unity in the case of asymptotic independence and therefore not informative, the extremal range can be used to quantify more precisely the degree of asymptotic dependence for asymptotically independent random fields. An important practical difference further stems from the fact that edge effects at the boundary of the observation domain play a more important role in the $d$-dimensional spatial setting than in the temporal one. We therefore focus on expressions that can be  formulated and properly estimated using the random field restricted to a compact observation window.

Our results are organized as follows. Section \ref{sec:defs} introduces the extremal range and related notations. In Section \ref{sec:parametrization}, we express the cumulative distribution function of the extremal range through the $d$ and $(d-1)$-dimensional Lipschitz-Killing  curvatures  of the excursion regions. In Section \ref{sec:model}, we study the asymptotic behavior of the extremal range for common random field models as the threshold at the conditioning location  is increased. 
Some technical definitions and examples are postponed to Appendix~\ref{sec:appendix}. Finally, proofs of results of Sections \ref{sec:parametrization} and~\ref{sec:model} are provided in Appendix~\ref{ProofsProvidedResults}.

\section{The extremal range and relevant notations}\label{sec:defs}

Let $(\Omega,\mathcal{F},\P)$ be a probability space and let $X:\Omega\times \Reals^{d} \rightarrow \Reals$ be a random field defined on $\Reals^{d}$, endowed with the Euclidean metric $\norm{\cdot}$. For a  domain $S\subseteq\Reals^{d}$, let $\partial S$ denote its topological boundary.  Furthermore, for a set $S\subset \Reals^{d}$,
let $\lebesgue_{d-1}(\partial S)$ denote the {surface area} of $S$ (\emph{i.e.}, the {$d-1$}-dimensional Hausdorff measure of its boundary) and  $\lebesgue_d(S)$ the volume of $S$ (\emph{i.e.}, the Lebesgue measure of $S$ on $\Reals^d$). 
 For $x\in\Reals^{d}$, denote the distance between $x$ and a non-empty set $S$ by $\dist(x,S):= \inf\{\norm{x-s} : s\in S\}$. 
Throughout this paper, $u:\Reals^{d}\to \Reals$  denotes a deterministic threshold function that is allowed to vary in space, and we focus on the binary random field of excursion indicators {$\{X(s) > u(s)\}_{s\in\Reals^{d}}$}. This is expressed in terms of the following definition.

\begin{defi}[Excursion   and level sets]\label{DefExcursionSet}
Let  $X$ be a random field on $\Reals^{d}$ and $u:\Reals^{d}\to \Reals$ be a    threshold function. Define the  excursion set of $X$ {for threshold $u$}  as 
$$E_X(u) := \{s\in\Reals^{d}:X(s) > u(s)\}.$$
 Similarly, we will consider in the present paper the  level surfaces   as   the boundary of $E_X(u)$, denoted 
 $\partial E_X(u)$. Furthermore,  the  excursion set and its boundary  when intersected with an observation window $T$ are  denoted by 
$E_{X}(u)  \cap T$ and $ \partial   (E_{X}(u) \cap T)$,  where   $T \subset \Reals^d$  is  a bounded  closed  set  with non empty interior.  
\end{defi}

The excursion set in Definition  \ref{DefExcursionSet} can be fully characterized by its capacity functional (see, e.g., \cite{Kratz_Nagel_2016}) i.e., $\P(E_X(u) \cap  K \neq \emptyset),$ for all compact  subsets $K \subset \Reals^d$. 

The following definition allows one to choose an appropriate threshold function for the random field $X$, such that threshold exceedances occur with positive probability. 

\begin{defi}[Upper bound]\label{upperBound}
Let $u_X^* : \Reals^d \to \R \cup \{\infty\}$ be a deterministic function defined by $s \mapsto \inf\{u\in\Reals: \P(X(s) > u) = 0\}$, so that $u_X^*(s)$ is the upper end-point of the support of the distribution of $X$ at location $s$.
\end{defi}

\begin{defi}[Extremal range]\label{def:extremal_range}
For $r>0$ and $s\in\Reals^{d}$, let {$B(s,r) := \{s'\in\Reals^{d}:\norm{s'-s} \leq r\}$} denote the closed {Euclidean} ball of radius $r$ centered at $s$. Let   {$\tilde R^{(u)} : \Omega \times \Reals^{d} \to [0,  +\infty]$} be a random field defined by 
{$$\tilde R^{(u)} (s) := \sup \left\{r\in\Reals^+ : B(s,r) \subset E_X(u)\right\} = \dist\big(s,(E_X(u))^c),\qquad s\in\Reals^{d},$$}
{where   $A^c$ stands for the complement set of set $A$.}
Let $s\in\Reals^{d}$ satisfy {$u(s) < u_X^*(s)$}, with $u_X^*(s)$ as in Definition \ref{upperBound}.  Define the \textit{extremal range} at $s$, denoted $R^{(u)}_s$, to be the {conditioned} random variable whose pushforward measure is given by
$$\P(R^{(u)}_s \in A) = \P\big(\tilde R^{(u)}(s) \in A  \,\vert\, X(s) > u(s)\big),\qquad A\in \mathcal{B}(\Reals),$$
{where $\mathcal{B}(\Reals)$ is the Borel $\sigma-$algebra of $\Reals$.}
\end{defi}

{Notice that the conditional variable $R^{(u)}_s$ exists only for $\omega\in\Omega$ with $X(s;\omega)>u$; therefore, it is defined on the probability space from which elements $\omega$ with $X(s;\omega)\leq u$ have been removed.}
{An illustration in dimension  $d=2$ of the excursion set $E_X(u)$ in Definition \ref{DefExcursionSet} and $\tilde R^{(u)}(s)$ in Definition \ref{def:extremal_range} is provided in Figure \ref{fig:excursion}.}

\begin{figure}[h!]
\includegraphics[width=0.9\linewidth]{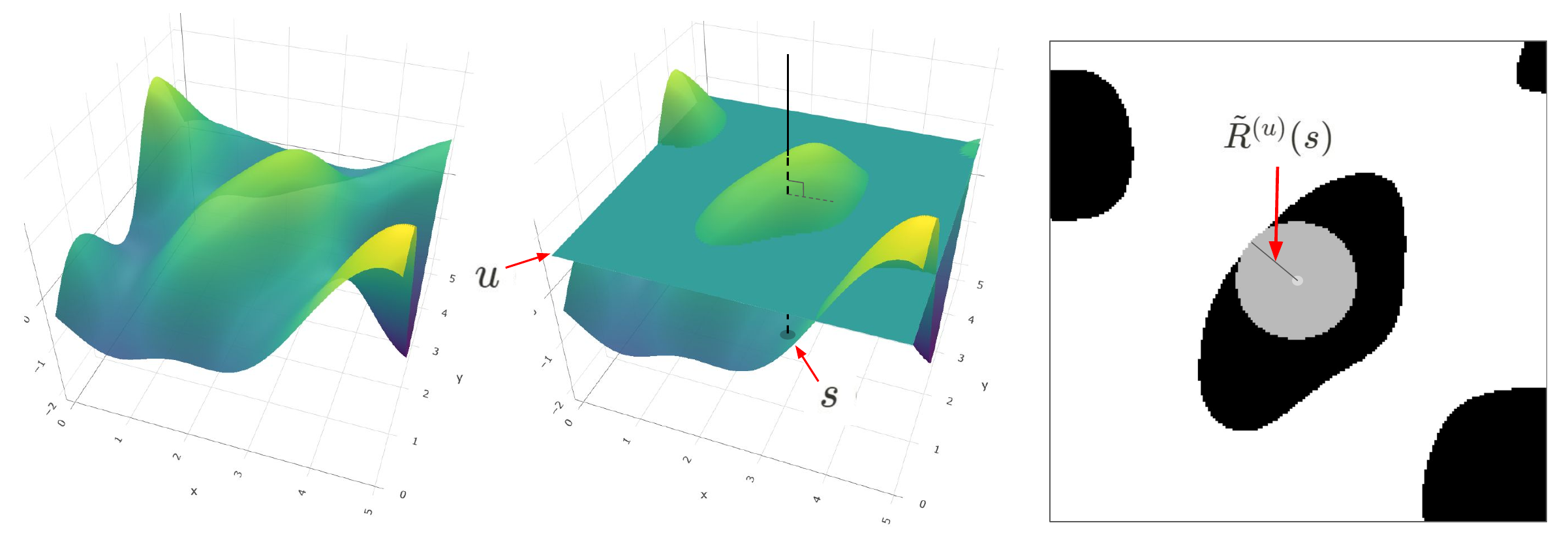}
    \caption{Example of an excursion set $E_X(u)\cap T$ with the quantity $\tilde{R}^{(u)}(s)$ shown for  a chosen $s\in T \subset \mathbb{R}^2$.}
    \label{fig:excursion}
\end{figure}

\begin{rema}\label{rem:extremal_index}
    As discussed by \cite{Moloney2019} in the time series context, the inverse of the so-called extremal index quantifies the average size of clusters of threshold exceedances, \emph{i.e.}, for one-dimensional discretely supported random processes. Analogously, the extremal range provides a notion of the size of the clusters of sites that exhibit threshold exceedances for continuous  random fields. 
\end{rema}

Definition \ref{def:minkowski} below is relevant to establish the main results for the extremal range.

\begin{defi}[Erosion and dilation]\label{def:minkowski}
For two nonempty sets $A, B \subseteq \Reals^{d}$, let $A \oplus B := \{x+y:x\in A, y\in B\}$ be the Minkowski sum of $A$ and $B$.
For $r\in\Reals$, and $S\subseteq \Reals^{d}$ let
$$S_r := \begin{cases}
        S\oplus B(0,r), & \text{for } r\geq 0,\\
        \big(S^c \oplus B(0,-r)\big)^c, & \text{for } r < 0,
        \end{cases} 
$$
denote respectively the set dilation and the set erosion, depending on the sign  of $r$.
\end{defi}
{To illustrate Definition \ref{def:minkowski}, in Figure \ref{fig:erosion} two types of erosion of $E_X(u)\cap T$ in dimension 2 are shown.} 
\begin{figure}[h!] 
    \centering
    \includegraphics[width=0.6\linewidth]{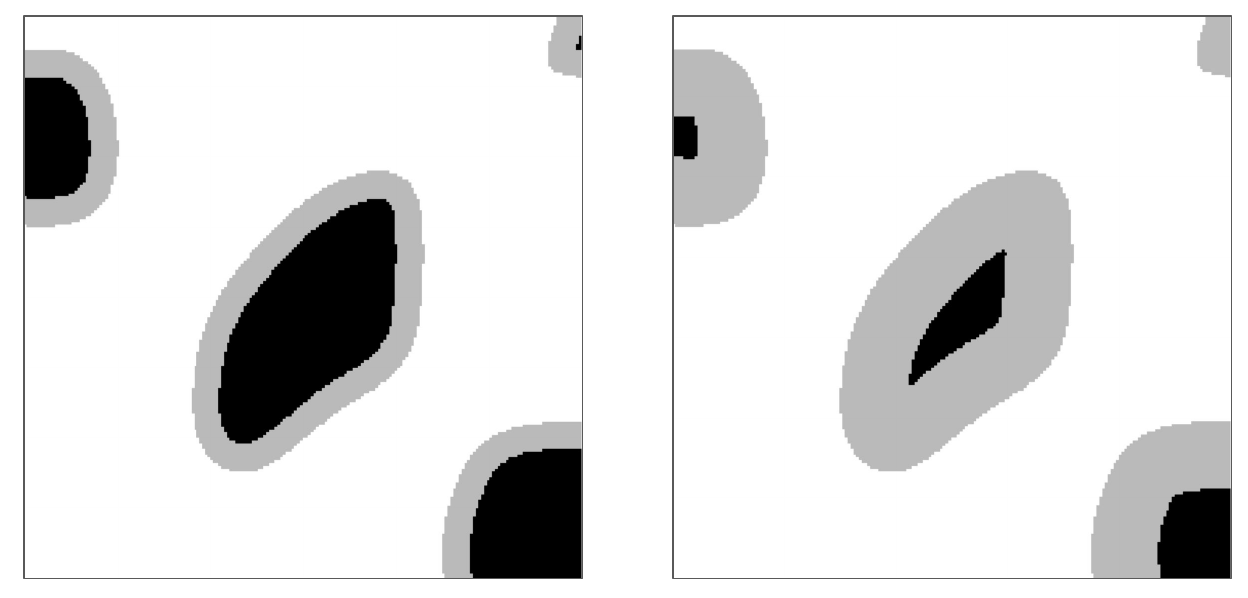}
    \caption{Two amounts of erosion of the excursion set $E_X(u)\cap T$ from Figure~\ref{fig:excursion}.}
    \label{fig:erosion}
\end{figure}
 
\begin{assu}\label{assumtion}
Let  $X$ be a random field on $\Reals^{d}$ and $u:\Reals^{d}\to \Reals$ be a    threshold function. 
 Suppose that for the random field $X$ paired with the threshold function $u$, the random 
 set $E_X(u)$ is stationary, and $E_X(u)^c \cap T$ has positive reach almost surely.     
\end{assu}

For sake of completeness, the definitions of     positive reach set  and stationary random set   
 are  recalled in  Definition~\ref{def:positive_reach} and  \ref{StationaryRandom setDef} respectively, see Appendix~\ref{sec:appendix_reach}.    
 
\begin{defi}[$d$ and $(d-1)$-dimensional Lipschitz-Killing curvature densities]\label{ass:regularity}
Under Assumption~\ref{assumtion}, for any  compact and   convex  
set $T \subset \Reals^d$ with non empty interior,  
one can define 
\begin{align} 
C_{d-1}^*(E_{X}(u)) &:= \lim_{n\to\infty} \frac{1}{2 }\frac{\E[\lebesgue_{d-1}(\partial(E_{X}(u)\cap nT))}{\lebesgue_d(nT)}, \label{haus}\\ 
C_{d}^*(E_{X}(u)) &:= \lim_{n\to\infty}\frac{\E[\lebesgue_d(E_{X}(u)\cap nT)]}{\lebesgue_d(nT)}, \label{volume} 
\end{align}
where $nT$ is the result after linearly rescaling $T$ by $n$. Under Assumption~\ref{assumtion},  
 quantities in \eqref{haus}-\eqref{volume}  exist, are finite, and independent of $T$. \end{defi} 
 
  Notice that if $X$ is  almost surely twice differentiable and $u$ a constant level function, then random set $E_{X}(u)^c \cap T$ has positive reach, as $E_{X}(u)^c$ is a $C^2$ sub-manifold of
$\Reals^d$ and its intersection with  $T$ provides compactness and positive
reach property (see \cite[Proposition~14]{thale2008}).   The finite quantities in \eqref{haus}-\eqref{volume}    correspond respectively to what we call in the sequel  half surface
area density (with a slight abuse of language), and $d$-volume 
  density.  Note that $C_i^*(E_X(u))$, for  {$i=d-1,d$},  {are two of} limiting normalized Lipschitz-Killing curvatures of the excursion set $E_X(u) \cap T$ seen on large domains (see e.g.,  Theorem~9.3.3 in \cite{schneider2008}). They play an important role in determining the shape of the distribution function of the extremal range $\tilde R^{(u)} (s)$ in Definition \ref{def:extremal_range}; a topic that we investigate in the following of the present work. 

\begin{rema}[Discussion of Assumption~\ref{assumtion}]
Notice  under Assumption~\ref{assumtion} the random field $X$ is not necessarily stationary, as $u$ is not necessarily a constant function in space. What is necessary instead is that the \textit{excursion set} at the level $u$ be stationary.  An example to illustrate this weak  condition   is given in Figure  \ref{fig:discontinuous} in Appendix \ref{Figure5Appendix}.  An important, easily verifiable consequence of this is that,   $C_{d}^*(E_X(u)) = \P(X(0) > u(0))$,  for $u$  as in  Assumption~\ref{assumtion}.  This  assumption also implies that $E_X(u) \cap T$ is almost surely open, as its complement must be closed to satisfy the positive reach property. Furthermore notice that Gaussianity is not a necessary condition for our results, except for Proposition~\ref{prp:gaussian_nondegen_limiting} specifically focusing on results for such fields. A final remark on the generality of Assumption \ref{assumtion} is that the almost sure continuity of $X$ and $X-u$ are not necessary (see \textit{e.g.}, the random field in Figure~\ref{fig:discontinuous} in Appendix \ref{Figure5Appendix} with a constant threshold $u \geq 1$).
\end{rema}

\section{Linking the extremal range and the Lipschitz-Killing  curvatures}\label{sec:parametrization}
{The following proposition relates the distribution function of the extremal range to the eroded excursion set observed  in $T$. More precisely, Proposition~\ref{prp:cdf} below states that the eroded excursion set $ E_X(u)_{-r}$ carries information about the distribution of $R_s^{(u)}$ through its volume when intersected with a compact set $T$.  }

\begin{proposition}\label{prp:cdf} 
Under Assumption~\ref{assumtion}, for any compact set $T\subset \Reals^d$ with $\lebesgue_d(T) > 0$, the distribution function of $R^{(u)}_0$ is given by
\begin{equation}\label{eqn:R_distn_func}
\P\big(R^{(u)}_0 \leq r\big) = 
    1-\frac{\E\big[\lebesgue_d\big(E_X(u)_{-r}\cap T\big)\big]}{\E\big[\lebesgue_d\big(E_X(u)\cap T\big)\big]},
\end{equation}
for $r\geq 0$, and $\P\big(R^{(u)}_0 < 0\big) = 0$, where the subscript $-r$ denotes set erosion by a radius of $r$ (see Definition~\ref{def:minkowski}).
\end{proposition}

The proof of Proposition~\ref{prp:cdf} is provided in Appendix~\ref{ProofsProvidedResults}. The extremal range has close links with the \textit{spherical erosion function}   \citep{serra1984, ripley1988}, which describes the distribution function of the distance of a uniform  random point in a set to the set's boundary.   The interested reader is referred to Section 1.7.4 in \cite{chiu2013stochastic}.   Notice that volumes   of excursion sets and their erosion can be efficiently estimated with routine algorithms, such that Equation~\eqref{eqn:R_distn_func} can be used for estimating the distribution function of $R_0^{(u)}$ in the stationary setting by replacing expectations with empirical estimates. 

Next in Theorem \ref{thm:distribution_small_r}, we will show that under certain regularity conditions, a polynomial expression of the Lebesgue measure of an eroded set in terms of the associated $d$ and $(d-1)$-dimensional Lipschitz-Killing curvatures (see Definition~\ref{ass:regularity}) can be obtained as corollary to the well-known Steiner formula \citep[][Theorem~5.6]{federer1959}.      An important preliminary  property of the extremal range  is asserted by the following lemma, which we prove in Appendix~\ref{ProofsProvidedResults}.

\begin{lemm}\label{lem:continuity}
    Under Assumption~\ref{assumtion}, it holds that    $\P(R_0^{(u)} \leq r)$ is continuous in $r$, for $r>0$.
\end{lemm}

The main result of this section is the following first-order approximation of the distribution function of the extremal range.   

\begin{theo}\label{thm:distribution_small_r}
Under Assumption~\ref{assumtion}, for $r>0$,  it holds that  \begin{equation}\label{eq:er-asymptotics-for-small-r}
    \lim_{r\to 0}\frac{\P(R_0^{(u)} \leq r)}r = \frac{2 C_{d-1}^*(E_X(u))}{C_{d}^*(E_X(u))}.
\end{equation} 
\end{theo}
The proof of Theorem~\ref{thm:distribution_small_r} can be found in Appendix~\ref{ProofsProvidedResults}.
Theorem~\ref{thm:distribution_small_r} shows that the distribution of the extremal range follows a first-order Taylor expansion for positive values of the radius $r$ near 0. Moreover, the linear coefficient is provided by the limit on the right-hand side of Equation~\eqref{eq:er-asymptotics-for-small-r}. By studying how this coefficient behaves for large thresholds, we gain insight about the spatial extent of high threshold exceedances. This point will be explored in the next section.

\section{Asymptotics for high thresholds and small distances}\label{sec:model}

We study the asymptotic behavior of the extremal range as the threshold function $u$ tends to the location-wise upper endpoint of the distribution of $X$ everywhere in space (see Definition \ref{upperBound}).
By studying the extremal range, we aim to capture information about the dependence structure of the random field $X$. Therefore, we use the threshold function $u_p:\Reals^{d}\to \Reals$ defined below as a location-wise quantile, such that it naturally adapts to the margins of the random field $X$, which is not necessarily stationary in our setting.

\begin{defi}\label{def:u_p}
    For $p \in (0,1)$ and a random variable $Y:\Omega\to\Reals$, let $q_{p}(Y)\in\Reals$ denote the $p$-quantile of $Y$, \textit{i.e.}, $q_{p}(Y) := \inf\{r\in\Reals : \P(Y \leq r) \geq p\}$. Now, define the adaptive threshold $u_p$ by the mapping $u_p(s) := q_p(X(s))$, for $s\in \Reals^{d}$.
\end{defi}

Theorem~\ref{thm:distribution_small_r} allows us to study how the extremal range decreases as the considered threshold increases, \textit{i.e.}, as $p\to 1$. This important result is summarized in the following corollary.

\begin{coro}\label{cor:scaling_as_p_to_infty}
Suppose that there exists $p_0 \in (0,1)$ such that for all $p\in(p_0,1)$, the random field $X$ paired with the threshold function $u_p$  satisfy Assumption~\ref{assumtion}. 
  Then a function $g:(0,1)\to \Reals$  satisfies
    \begin{equation}\label{eqn:g_asymptotic_equivalence}
        \lim_{p\to 1} g(p)\,\frac{C_{d}^*(E_X(u_{p}))}{2C_{d-1}^*(E_X(u_{p}))} = \frac 1K,
    \end{equation}
    for some $K\in\Reals^+$, if and only if 
\begin{equation}\label{eqn:lim_p_lim_r}
        \lim_{p\to1}\lim_{r\to 0}\frac{\P(g(p)R_0^{(u_p)} \leq r)}r = K.
    \end{equation}
\end{coro}
\begin{proof}
    Theorem~\ref{thm:distribution_small_r} tells us that for any $p$,
    $$\lim_{r\to 0}\frac{\P(g(p)R_0^{(u_p)} \leq r)}r = \frac{2C_{d-1}^*(E_X(u_{p}))}{g(p)\, C_d^*(E_X(u_{p}))}.$$
    Sending $p\to 1$ yields the desired result.
\end{proof}
As we will see in the following, the asymptotics of $g(p)$ can be derived for random fields of \textit{Gaussian type} from the Gaussian Kinematic Formula (see  Theorem~15.9.5 in \citep{adler2007}). For instance, Gaussian random fields satisfy Equations~\eqref{eqn:g_asymptotic_equivalence} and~\eqref{eqn:lim_p_lim_r} when $g(p) \asymp \sqrt{-\log(1-p)}$, where the binary operator $\asymp$ indicates that the ratio of the two expressions tends to a constant as $p\to 1$. Likewise, for so-called \textit{regularly varying} random fields, we will show that $g(p)\asymp 1$ is necessary and sufficient for Equations~\eqref{eqn:g_asymptotic_equivalence} and~\eqref{eqn:lim_p_lim_r} to hold (the details of regularly varying fields will be discussed in Section~\ref{sec:model_nondegen_rv}).

An interpretation of Corollary~\ref{cor:scaling_as_p_to_infty} is that the probability density function of $g(p)R_0^{(u_p)}$ just to the right of 0 approaches 1 if and only if $g(p)$ is asymptotically equivalent to $2C_{d-1}^*(E_X(u_{p})) / C_{d}^*(E_X(u_{p}))$ as $p\to 1$. In this sense, Corollary~\ref{cor:scaling_as_p_to_infty} shows how $R_0^{(u_p)}$ scales as $p\to1$. We are not able to use Corollary~\ref{cor:scaling_as_p_to_infty} to  {more generally} establish a non-degenerate limit distribution of $\big(2C_{d-1}^*(E_X(u_{p}))/C_{d}^*(E_X(u_{p}))\big)R_0^{(u_p)}$ as $p\to 1$; it is not always possible to exchange the order of the limits in Equation~\eqref{eqn:lim_p_lim_r}. A counterexample in dimension $d=2$ is provided in Appendix~\ref{sec:appendix_counter}.
 
\subsection{Non-degenerate limit distributions of the extremal range}\label{sec:model_nondegen}

Here, we study certain cases of widely used spatial random field models where the extremal range is known to have a non-degenerate limit distribution at high thresholds \textit{after appropriate rescaling}.  We give expressions for the rates of change of the extremal range at extremely large conditioning thresholds.   
The random fields that we will consider in
this section
are stationary, so we choose a threshold function $u$ that is constant throughout space. To ease notation, we write $u$ to denote both the constant mapping $u:\Reals^{d}\to\Reals$ and its image in $\Reals$.

\subsubsection{Gaussian random fields}\label{sec:model_nondegen_gaussian}

For a smooth, stationary Gaussian process $Y$ on $\Reals$, if one is to condition on the event $\{X(0) > u\}$ for some large threshold $u\in\Reals$, one can show using tools developed in \cite{kac1959} that the connected component of the excursion set containing $0$ is a random interval with expected length asymptotically equivalent to $1/u$. By analogy, after appropriately rescaling in the spatial dimension, one finds that the limit process is a random parabola with deterministic shape. These insights are formally generalized for the $d$-dimensional case in the following proposition formulated for smooth standard Gaussian fields, for which a proof is given in Appendix~\ref{ProofsProvidedResults}.

\begin{proposition}\label{prp:gaussian_nondegen_limiting}
    Suppose that $X$ is a stationary, isotropic, centered Gaussian random field on $\Reals^{d}$ with covariance function
    \begin{equation}\label{eqn:gaussian_correlation}
        \rho(h) = 1 - \frac{\lambda}{2}  \|h  \|^2 + o(  \|h  \|^2),\qquad \lambda> 0,
    \end{equation}
    for $h$ in a neighbourhood of 0.
    Then $\P(uR^{(u)}_0 \in \cdot )$ converges to a non-degenerate probability distribution, as $u\to \infty$.
\end{proposition}
Notice that, for a  $d$-dimensional   random field $X$ as described in Proposition~\ref{prp:gaussian_nondegen_limiting}, the expressions for $C_{d-1}^*(E_X(u))$ and $C_{d}^*(E_X(u))$ are computed in \cite{bierme2019} using the Gaussian Kinematic Formula \citep[][Theorem~15.9.5]{adler2007}. The interested reader is also referred to Exercises 6.2.c and 6.3 in \cite{AW09}. Then we have,   
$$\frac{2C_{d-1}^*(E_X(u))}{C_d^*(E_X(u))} = \sqrt{\frac{\lambda}{\pi}} \, \frac{e^{-u^2/2}}{1-\Phi(u)}   \,\frac{\Gamma(\frac{d+1}{2})}{\Gamma(\frac{d}{2})},
$$

where $\Phi$ denotes the standard Gaussian cumulative distribution function, and $\lambda$ is as in~Equation \eqref{eqn:gaussian_correlation}. Therefore, by the results of \cite{gordon1941}, concerning the Mill's ratio of the Gaussian  distribution,
$$\lim_{u \to \infty} \frac{2C_{d-1}^*(E_X(u))}{C_d^*(E_X(u))}\times \frac 1u =  \frac{\sqrt{2\lambda}\,\Gamma(\frac{d+1}{2})}{\Gamma(\frac{d}{2})}.$$

Therefore, by using Corollary~\ref{cor:scaling_as_p_to_infty} with $g(p)=u_p(0)$, the probability density of $uR_0^{(u)}$ just to the right of $0$ approaches {$\frac{\sqrt{2\lambda}\,\Gamma(\frac{d+1}{2})}{\Gamma(\frac{d}{2})}$}
as $u\to\infty$.
If one were to show in addition the uniform convergence of the density of the extremal range, one may conclude that {$\frac{\sqrt{2\lambda}\,\Gamma(\frac{d+1}{2})}{\Gamma(\frac{d}{2})}$}
is the limiting value as $r\to0$ of the limiting density, as $u\to \infty$.  Possible  Gaussian covariance functions satisfying Equation~\eqref{eqn:gaussian_correlation} are given below.

\begin{exam}
\label{example1}
A stationary, isotropic Gaussian random field with unit variance and $C^1$-smooth sample paths has the covariance function in~\eqref{eqn:gaussian_correlation} with $\lambda$ equal to its second spectral moment; see \cite[page~151]{Leadbetter1983} and \cite{cambanis1973}.
A particular case is given by the  isotropic, Mat{\'e}rn covariance function
\begin{equation}\label{Matern}
    \rho(h) = \frac{2^{1-\nu}}{\Gamma(\nu)}\left(\frac{\sqrt{2\nu}\norm{h}}{l}\right)^\nu K_\nu\left(\frac{\sqrt{2\nu}\norm{h}}{l}\right), \quad \nu,l>0,
\end{equation}
with $K_\nu$ denoting the modified Bessel function of the second kind, satisfies~\eqref{eqn:gaussian_correlation} for $\nu > 1$ and
$\lambda = \frac{\nu}{l^2(\nu - 1)}.$
\end{exam} 
In practice, a useful approximation of the distribution function of the suitable scaled $R_0^{(u)}$, for large $u$ and small $r$, for $d$-dimensional  Gaussian random fields is given by  
$$
\P(u R_0^{(u)}\leq r) \approx 
{\frac{\sqrt{2\lambda}\,\Gamma(\frac{d+1}{2})}{\Gamma(\frac{d}{2})}} \,r,
$$

where an estimate $\hat{\lambda}$ of the second spectral moment $\lambda$ based on a parametric covariance function $\rho(h)$  {or on the Euler-Poincar{\'e} 
 characteristic of  the excursion set  (see, e.g., \cite{bierme2019}, Proposition 2.6)} could be plugged in to obtain an estimate for spatial data corresponding to relatively smooth spatial surfaces.   

\subsubsection{Regularly varying fields}\label{sec:model_nondegen_rv}
 
 {Regular variation is a key concept to describe asymptotic behavior in tails of probability distributions \citep{Bingham1989}. Here, we recall the core elements of the theory of regularly varying random fields \citep{Hult2005,Hult2006}, and the related $\ell$-Pareto limit processes from \cite{ferreira2014,dombry2015}, commonly used as statistical models for spatial processes conditioned on high threshold exceedances of a certain {\textit cost (or loss) functional} $\ell$.  In practice, $\ell$-Pareto limit processes are very useful since one can tailor $\ell$ to the type of extreme events one is interested in, they are generative models one can simulate from, and one can extract their realizations from real data for statistical modeling.} 
 
Let $T$ be a compact domain 
satisfying $r_T := \sup\{r\in \Reals^+ : B(0, r)\subseteq T\} > 0$. Let $X$ be a continuous, stationary random field defined on $\Reals^d$, and let $X |_{T}$ be the random field $X$ restricted to the domain $T$. Let $\mathcal{C}_0$ be the set of continuous functions from $T$ to $[0,\infty)$, excluding the constant function $0$. Let $\mathcal{S} = \{x\in \mathcal{C}_0 :   \|x  \|_T = 1\}$, where $  \|x  \|_T := \sup_{s\in T} x(s)$.

In Appendix~\ref{sec:appendix_rv}, we recall from  \cite{dombry2015} what it means for a random field to be \textit{regularly varying with exponent $\alpha$ and spectral measure $\sigma$ on $\mathcal S$}. 
The limiting behavior of these random fields at high thresholds can be well described by \textit{$\ell$-Pareto random fields} (see Lemma~\ref{lem:rvlp} in Appendix~\ref{sec:appendix}); more recently also called $r$-Pareto random fields with $r$ standing for \emph{risk} \citep{fondeville2022}.
These random fields are constructed by using a cost functional $\ell : \mathcal{C}_0 \to [0,\infty)$, and the two $\ell$-Pareto processes that we study in relation to the extremal range use the cost functionals $\ell_0(x) = x(0)$ (i.e., conditioning on an exceedance at the origin) and $\ell_T(x) = ||x||_T$ (i.e., conditioning on an exceedance of the spatial supremum). The precise definition of an $\ell$-Pareto random field is provided in Definition~\ref{def:l-pareto} in Appendix~\ref{sec:appendix}.  {For regularly varying $X$, it is possible to express the limit distribution of the extremal range in terms of these two different constructions of $\ell$-Pareto processes.}

\begin{proposition}\label{prp:regularly_varying_extremal_range}
    Suppose that $X |_T$ is regularly varying with exponent $\alpha > 0$ and spectral measure $\sigma$ on $\mathcal{S}$. Let $Y_0$ and $Y_T$ be $\ell$-Pareto processes with exponent $\alpha$ and respective spectral measures $\sigma_0$ and $\sigma_T$, as defined in Appendix~\ref{sec:appendix_rv}.
    Then, for $r \in (0,r_T)$,
    \begin{equation}\label{eqn:rv_limits}
        \lim_{u\to\infty}\P(R^{(u)}_0 \leq r) = 1 - \frac{\E\left[\lebesgue\big((E_{Y_T}(1) \cap T)_{-r}\big)\right]}{\lebesgue_d(T_{-r})\P(Y_{T}(0) > 1)} = \P\big(\exists t\in B(0,r) \mbox{ s.t. } Y_0(t) \leq 1\big).
    \end{equation}
\end{proposition}
\smallskip

The proof of Proposition~\ref{prp:regularly_varying_extremal_range} is postponed to  Appendix~\ref{ProofsProvidedResults}. The first equation gives an expression of the distribution function \eqref{eqn:R_distn_func} using the properties of the Pareto limit random fields.  {The statement of Proposition~\ref{prp:regularly_varying_extremal_range} tells us that for large enough thresholds $u$, regularly varying fields possess approximately the same distribution of their extremal range whatever the exact value of $u$. The asymptotic behavior can therefore be characterized by normalizing with respect to the threshold value and using a fixed threshold $u=1$, without explicitly conditioning on exceedance in the $\ell$-Pareto limit processes, since $Y_0(0)>1$ and $\|Y_T\|_T>1$ almost surely. In practice, the limit distribution of $R^{(u)}_0$ does not depend on $u$, and can therefore be estimated by pooling information from excursion sets at several high thresholds $u$.}

\subsection{Connections with the tail dependence coefficient}

Taking a more non-parametric perspective, we continue using the threshold function $u_p$ as defined in Definition~\ref{def:u_p} that adapts to non-stationary random fields. Recall that for two sites $s_1, s_2 \in \Reals^{d}$, the tail dependence coefficient function of a spatial random field $X$ is defined as $\chi(s_1,s_2) := \lim_{p\to 1}\chi_p(s_1,s_2)$, where, \begin{equation}\label{chiEq}\chi_p(s_1,s_2) := 
\P\left(X(s_1) > u_p(s_1) \ \vert X(s_2) > u_p(s_2)\right).
\end{equation}

The interested reader is referred to \cite{LI2009243} for possible generalisations of the tail
dependence coefficient   function  in~\eqref{chiEq} in the  multivariate setting. 
Here, we use the following definition of asymptotic (in)dependence. The random field $X$ is said to be \textit{asymptotically independent} if $\chi(s_1,s_2)= 0$ for all $s_1 \neq s_2$, and \textit{asymptotically dependent} if $\chi(s_1,s_2)> 0$, for all $s_1, s_2 \in \Reals^{d}$.
If $X$ exhibits asymptotic independence, then we have immediately that $R_s^{(u)}\xrightarrow[]{\P} 0$, as $u\to\infty$. This simple observation is a corollary of the following proposition.

\begin{proposition}\label{prp:dependence_extremal_range_relation}
    Let $X$ be any random field on $\Reals^d$. For all $s\in\Reals^{d}$ and all $p\in (0,1)$,
    $$\P\big(R_0^{(u_p)} \leq \norm{s}\big) \geq 1 - \chi_p(s,0).$$
\end{proposition}
\begin{proof}
 {For all $p \in (0, 1)$,} the event $\{\tilde R^{(u_p)}(0) > \norm{s},\, X(0) > u_p\}$ is contained in the event $\{X(s) > u_p,\, X(0) > u_p\}$. Therefore,
    $\P\big(\tilde R^{(u_p)}(0) > \norm{s},\, X(0) > u_p\big) \leq \P\big(X(s) > u_p,\, X(0) > u_p\big)$. A division by $\P(X(0) > u_p)$ (equal to $1-p$ if $X(0)$ has a continuous distribution function) implies
    $\P\big(\tilde R^{(u_p)}(0) > \norm{s} \ |  
    X(0) > u_p\big) \leq \P\big(X(s) > u_p \ | X(0) > u_p\big)$,
    and the result holds by taking compliments.
\end{proof}

Therefore, asymptotic dependence is a necessary condition for $R^{(u_p)}_0$ to have a non-degenerate limit distribution as $p\to 1$, with no rescaling. However, it is not sufficient in general. In Appendix~\ref{sec:appendix_counter}, we construct a malicious example of an asymptotically dependent random field for which $R^{(u_p)}_0 \xrightarrow[p\to 1]{\P} 0$.

The following theorem makes an important link between the extremal range and the tail dependence coefficient, and establishes that in this specific case, $\chi(0,s) = 1$ for all $s\in\Reals^{d}$.

\begin{theo}\label{theo2}
Let $X$ be an isotropic random field on $\Reals^d$. Suppose that there exists $p_0 \in (0,1)$ such that for all $p\in(p_0,1)$, the random field $X$ paired with the threshold function $u_p$  satisfy Assumption~\ref{assumtion}.  Let $h$ be a real function of $p\in (0,1)$ such that
    $$h(p)\frac{C_d^*(E_X(u_p))}{C_{d-1}^*(E_X(u_p))}\xrightarrow[p\to1]{} \infty.$$
    Then, for any fixed $s\in\Reals^d$,
    $\chi_p(s/h(p), 0) \xrightarrow[p\to 1]{} 1$.
\end{theo}
\begin{proof}
    Under Assumption~\ref{assumtion} one can use the result of \cite[Theorem~2.1]{cotsakis2022_1} which states that for $p\in (0,1)$ and $s\in\Reals^d$,
    \begin{equation}\label{eqn:beta_d}
        \frac{1}q\P\big(X(qs) \leq u_p < X(0)) \xrightarrow[q\to 0]{} \frac{2 C_{d-1}^*(E_X(u_p))}{\beta_d}\, \norm{s},
    \end{equation}
    where $
\beta_d = \frac{2\sqrt{\pi}\ \Gamma(\frac{d+1}{2})}{\Gamma(\frac{d}{2})}$ and the limit is approached from below.  
    Thus, for any $q\in \Reals^+$, a division by $1-p$ yields
    $$\frac{1-\chi_p(qs,0)}q \leq \frac{2C_{d-1}^*(E_X(u_p))}{\beta_d\, C_d^*(E_X(u_p))} \norm{s}.$$
      By setting $q = 1/h(p)$, we find that
    $$1-\chi_p(s/h(p),0) \leq \frac{2 C_{d-1}^*(E_X(u_p))}{h(p)\beta_d\, C_d^*(E_X(u_p))} \norm{s} \xrightarrow[p\to 1]{} 0.$$
    The desired result holds since $\chi_p \in [0,1]$ for all $p\in (0,1)$.
\end{proof}

\begin{rema}
    Recall that $2\,C_{d-1}^*(E_X(u_p))/C_d^*(E_X(u_p))$ is the limit value in Theorem~\ref{thm:distribution_small_r}, giving the first-order approximation of the cumulative distribution function of the extremal range. For many random fields, this is seemingly the rate at which space should be rescaled as $p\to 1$ if one is to expect the tail dependence coefficient and the distribution function of the extremal range to stabilize to values strictly between $0$ and $1$. 
\end{rema}
     
\section{Numerical studies}\label{numericSection}  
This section presents two numerical studies to illustrate the main results of this paper, namely Theorems~\ref{thm:distribution_small_r} and~\ref{theo2}. To limit the computational complexity of these studies, all of the random fields that we simulate are two dimensional ($d = 2$).

To illustrate the generality of our results, we consider a variety of random fields for which the Lipschitz-Killing Curvatures densities  have known with closed-form expressions. The Gaussian Kinematic Formula \citep[][Theorem~15.9.5]{adler2007} can be leveraged to obtain the LKCs densities for random fields of \textit{Gaussian type}. Such random fields can be expressed in terms of a functional of finitely many independent, identically distributed Gaussian random fields. The two-dimensional Gaussian random fields that we consider in this section, denoted by $G$, are stationary with zero-mean, unit-variance, and Matérn covariance function (see Equation  \eqref{Matern})
with $\nu = 2.5$. 

The four types of random fields that we construct from the distribution $G$ are obtained as follows.
\begin{itemize}
    \item \textbf{Gaussian random field $G$.} The two-dimensional random field $G$ is described above.
    \item \textbf{Student random field $T$.} For independent realizations $G_1, G_2, G_3, G_4$ of $G$, we consider the Student random field with $k = 3$ degrees of freedom, defined by     $T(x) := G_4(x)/ \sqrt{G_1(x)^2 + G_2(x)^2 + G_3(x)^2}$, for any $x \in\Reals^2$.
    \item \textbf{$\boldsymbol{\chi}^2$ random field $K$}. For independent realizations $G_1, G_2, G_3$ of $G$, we consider the $\chi^2$ random field with $k = 3$ degrees of freedom, defined by  $K(x) :=  G_1(x)^2 + G_2(x)^2 + G_3(x)^2$,  for any $x \in\Reals^2$.
    \item \textbf{Gaussian scale mixture random field $W$.} With $\Lambda$ a Pareto random variable with parameter $\alpha=2$,  independent of $G$, we consider the Gaussian mixture random field  defined by $W(x) :=  \Lambda \times G(x)$,  for any $x \in\Reals^2.$
\end{itemize}
For specific formulas of local maxima and the expected Euler characteristic of
excursion sets of $\chi^2$ and Student random fields above, the interested reader is referred to  \cite{Worsley_1994}. 

The random field $G$ can be efficiently simulated on a discrete set of points by computing the covariance for each pair of points in the set, and performing a Cholesky decomposition of the resulting covariance matrix. The resulting matrix can be then used to transform a vector of standard Gaussian random variables to realizations of $G$ evaluated on the points in the domain. In this case study, $G$ is simulated on a square grid of $121 \times 121$ equispaced points with vertices at $(-0.5, -0.5)$ and $(0.5, 0.5)$.

\begin{figure}[H]
    \centering
\includegraphics[width=0.9\linewidth]{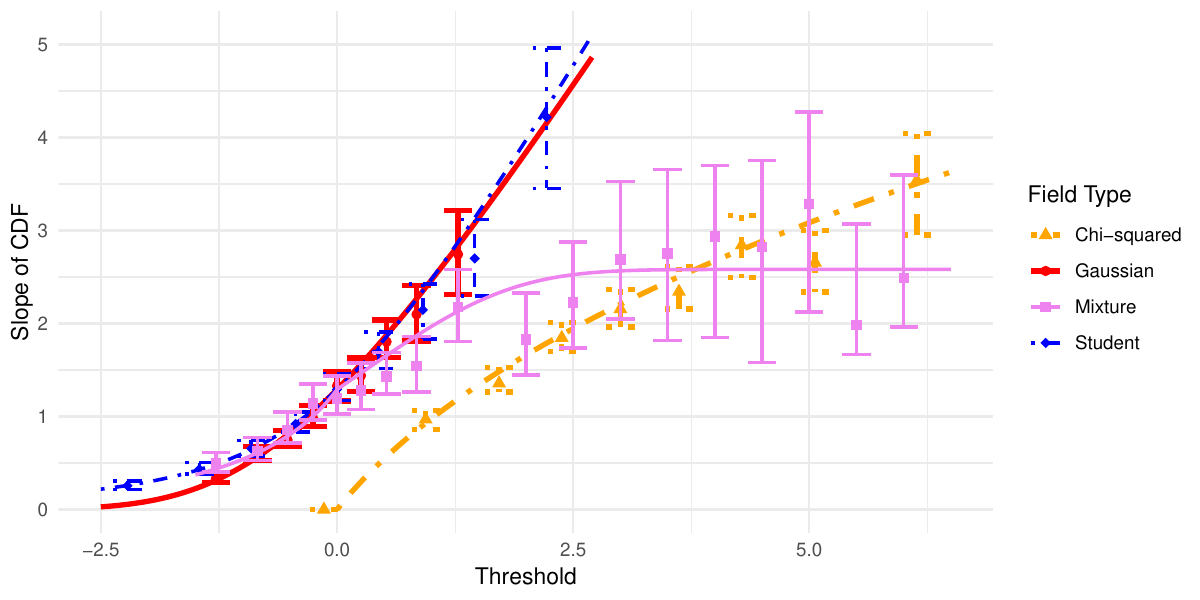}
    \caption{For each random field type, the quantity $\lim_{r\to 0^+}\P(R_0^{(u)} \leq r)/r$ is estimated for several values of the threshold $u$. These estimates and their 95\% confidence intervals are shown along with the curves that correspond to the theoretical values.}
    \label{fig:thm1}
\end{figure}

\begin{figure}[H]
    \centering
    \includegraphics[width=0.9\linewidth]{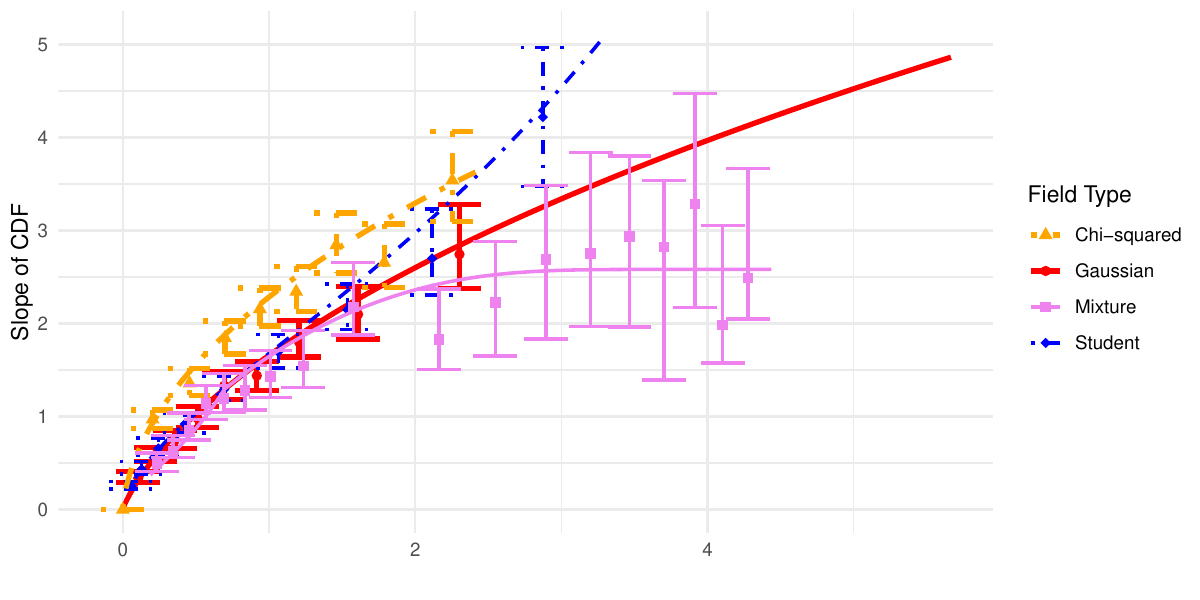}
    \put(-210,3){$-\log(1-p)$}
    \caption{See the caption of Figure~\ref{fig:thm1}. The same information is displayed, but the random fields are normalized to have $\mathrm{Exp}(1)$ margins. That is, the threshold corresponding to $p\in (0,1)$ is $u_p$ as defined in Definition~\ref{def:u_p}.
    }
    \label{fig:thm1_standardized}
\end{figure}

\subsection{Probability density of the extremal range near 0}

Theorem~\ref{thm:distribution_small_r} provides the limiting value of the slope of the cumulative distribution function (CDF) of the extremal range at 0 for a fixed threshold $u$. Since Assumption~\ref{assumtion} holds for all of the aforementioned random fields, for any finite $u$, Theorem~\ref{thm:distribution_small_r} can be applied. These theoretical values are compared with empirical estimates, for several threshold values $u$, in Figure~\ref{fig:thm1}. We simulated 5000 independent realizations of each of the above mentioned random fields. For a series of predetermined thresholds, the empirical cumulative distribution function of the extremal range was calculated based off of the realizations that exceeded the threshold at the origin {distance to first non-exceedance}. The slope of the CDF at 0 was estimated by fitting a smooth spline through the empirical CDF, and computing the derivative of the spline. The error bars, reflecting the 95\% confidence interval of the estimate, were computed by bootstrapping over the realizations of $G$ that surpassed the threshold at the origin, and computing the empirical CDF for each bootstrap.

Additionally, the same results from Figure~\ref{fig:thm1} are presented on standard exponential margins in Figure~\ref{fig:thm1_standardized}. That is, the threshold $u_p$ is used for each of four random fields of Gaussian type. This allows one to compare the different rates at which asymptotic dependence is lost at higher thresholds ($p$ close to 1), as measured by the quantity $\lim_{r\to 0^+}\P(R_0^{(u_p)} \leq r)/r$. While Figures~\ref{fig:thm1} and~\ref{fig:thm1_standardized} provide numerical justification to the theory established in Section~\ref{sec:parametrization}, they also offer insights into the differing rates at which spatial dependence changes as various thresholds are chosen. Intuitively, the typical size of threshold exceedences are inversely proportional to the slope of the CDF of the extremal range evaluated near 0. The four different random fields exhibit different behaviours as seen in Figures~\ref{fig:thm1} and~\ref{fig:thm1_standardized}.

It is clear that for the case of the Gaussian scale mixture, the behaviour of the distribution of the extremal range near 0 stabilizes for large thresholds. In Figure~\ref{fig:thm1_standardized}, we see that this stability is reached starting from as low as the 95\% quantile of the marginal distribution. This highlights an encouraging feature of the extremal range, is that the asymptotic behaviour is reached quite quickly, and so approximating by limiting behaviour may be justified in practical applications when large, but reasonable threshold levels are used. By Proposition~\ref{prp:dependence_extremal_range_relation}, the stability seen in the case of the Gaussian scale mixture field $W(x)$, for $x \in \Reals^2$, implies that $W$ must be asymptotically dependent in the traditional sense.

When comparing the rate at which asymptotic independence is achieved in the three asymptotically independent models, we see that the Student field $T$, has much weaker spatial dependence than the $\chi^2$ field $K$ at high thresholds. The Gaussian random field $G$, with its weak margins, has spatial dependence comparable to $T$ when absolute thresholds are considered in Figure~\ref{fig:thm1}, but its spatial dependence is comparable to that of $K$ when the threshold is chosen at the same quantile of the respective marginal distribution.

\subsection{The conditional probability of exceedance as a function of the distance to the conditioning site}

Figure~\ref{fig:thm2} illustrates the rate at which asymptotic dependence is lost at high thresholds for the aforementioned models, once normalized to have standard exponential margins. For various thresholds on these margins (shown on the $x$-axis of the figure), we compute both the theoretical and empirical slope of the function $f:\R \to \R$ defined by $f_p(r) = \chi_p(0, re_1)$, where $e_1$ is the unit vector $(1,0)$, and $\chi_p$ is defined in Equation~\eqref{chiEq}. The theoretical value of $f_p'(0)$ can be obtained from Equation~\eqref{eqn:beta_d} (with $d=2$) by a division of $C_2^*(E_X(u_p))$, yielding
$$f_p'(0) = -\frac{2C_{1}^*(E_X(u_p))}{\pi \, C_2^*(E_X(u_p))}.$$

The empirical value of $f_p'(0)$ is obtained for several $p$ as follows. For each random field, and for each value of $p$ tested, 500 random fields that exceed the $p$-quantile at the origin are generated. For each of these random fields, the proportion of sites that exhibit threshold exceedances within a ball of radius $r$ is calculated for several values of $r$. By averaging over the 500 realizations, we obtain an estimate of 
$$\phi_p(r) := \frac{\E[\lebesgue_2(E_X(u_p) \cap B(0,r))\ |\ X(0) > u_p(0)]}{\lebesgue_2(B(0,r))}$$
for several $r$. In our two-dimensional, isotropic setting, the quantities $\phi_p(r)$ and $f_p(r)$ can be shown to be related by
\begin{equation}\label{eqn:phi_and_f}
    f_p(r) = \phi_p(r) + \frac r2 \phi_p'(r).
\end{equation}
Thus, by fitting a smooth cubic spline to the empirical values of $\phi_p(r)$ and applying Equation~\eqref{eqn:phi_and_f}, one obtains estimates for $f_p$ and its first three derivatives. By bootstrapping over the 500 random fields that exceed the $p$-quantile at the origin, we obtain the 95\% confidence interval from the distribution of the resulting estimates of $f_p'(0) = \frac{3}{2} \phi'_p(0)$.

In Figure \ref{fig:thm2}, we display the slope of the empirical tail dependence coefficient at 0 for different values of $u$. The negative reciprocal of $f'_p(0)$ provides an evolving distance (as a function of $p\in(0,1)$) between the test site and the conditioning site, such that the conditional probability of exceedance stabilizes to a value in $(0,1)$ as $p\to 1$.

\begin{figure}
    \centering
    \includegraphics[width=0.9\linewidth]{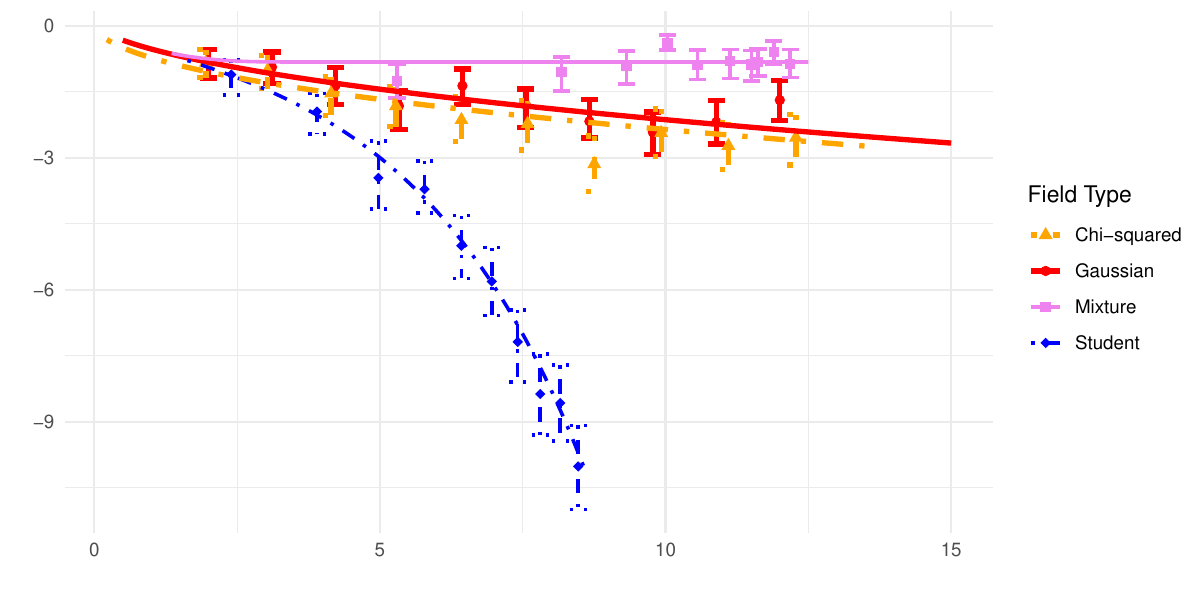}
    \put(-210,3){$-\log(1-p)$}
    \put(-350,100){$f'_p(0)$}
    \caption{For each random field type, the quantity $f'_p(0)$ is estimated for several values of $p \in (0,1)$. These estimates and their 95\% confidence intervals are shown along with the curves that correspond to the theoretical values.}
    \label{fig:thm2}
\end{figure}

Remark that each of the theoretical curves shown in Figure~\ref{fig:thm2} are the same as those in Figure~\ref{fig:thm1_standardized}, only rescaled by a negative constant. This reinforces the fact that both the distribution of the extremal range and the tail dependence function are deeply linked to each other by the Lipschitz-Killing curvature densities.

\section{Conclusion}
The extremal range proposed in this paper quantifies the degree of asymptotic (in)dependence locally at the site $s \in \Reals^d$. It is intended to aid flexible  exploratory analysis of dependence in spatial and spatiotemporal extremes beyond the mathematically elegant but rigid framework of asymptotically stable and stationary dependence in max-stable processes and other regularly varying processes. The extremal range can be seen as a  tool at the intersection of spatial extreme value theory, stochastic geometry and topological data analysis. Further research in this area could help foster a high-dimensional statistical learning toolbox for studying complex structures in large data volumes, especially in climate data. \bigskip

\bibliographystyle{APT}
\bibliography{biblio}

\appendix
\section*{Appendix}
\renewcommand\thesection{A}
\section{Technical definitions and examples}\label{sec:appendix}

\subsection{Stationary and positive reach sets}\label{sec:appendix_reach}

\begin{defi}\label{def:positive_reach}
\citep{federer1959} The reach of a set $S\subseteq\Reals^{d}$ is given by
$\sup\{r \in \Reals: \forall x \in S_r,\, \exists ! s\in S\mbox{ nearest to } x\}$, with $S_r$ as in Definition \ref{def:minkowski}.  A subset of $\Reals^{d}$ is termed \textit{positive reach} if its reach is positive.
\end{defi}

Recall from \cite{federer1959} that a closed set is convex if and only if its reach is infinite. Therefore, the empty set is trivially a positive reach set.

\begin{defi}\label{StationaryRandom setDef} \cite[][Section 6.1.4]{chiu2013stochastic} A random closed set $\,\Xi$ is said to be stationary if $\,\Xi$  and the translated set $\,\Xi_x = \Xi +x$ have the same distribution, for any $x \in  \Reals^{d}$. 
\end{defi}
\subsection{Further  details about  Assumption~\ref{assumtion}}\label{Figure5Appendix}

Under Assumption~\ref{assumtion}, the random field $X$ is not necessarily stationary, as $u$ is not necessarily a constant function in space.    An example  in dimension $d=1$ to illustrate this   condition   is given in Figure  \ref{fig:discontinuous} below.  We display several realizations of a one-dimensional random process $X$ which is non-stationary and not almost surely continuous, but whose excursion set $\{t\in\mathbb{R} : X(t) > u\}$ is a stationary random set for all fixed $u \geq 1$ (see Definition \ref{StationaryRandom setDef}). 

\begin{figure}[H] 
    \centering
    \includegraphics[width=0.7\linewidth]{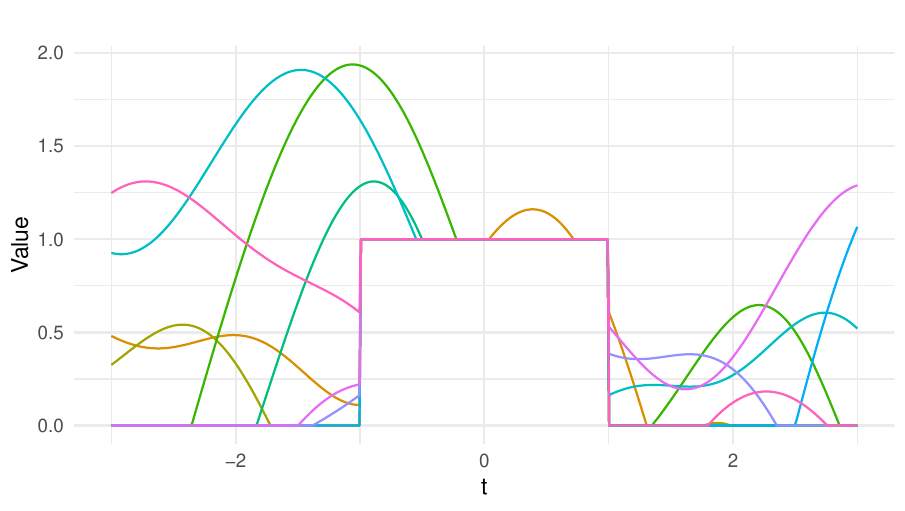}
    \caption{Several realizations, displayed with different colours, of $X$, the point-wise maximum of a stationary Gaussian process with covariance function $C(h) = e^{-h^2/2}$, and the indicator function \(\I{ \vert t \vert \leq 1}\).}
    \label{fig:discontinuous}
\end{figure}

\subsection{A  case where asymptotic dependence is not captured by the extremal range}\label{sec:appendix_counter}

\begin{figure}[H] 
    \centering
    \includegraphics[width=0.5\linewidth]{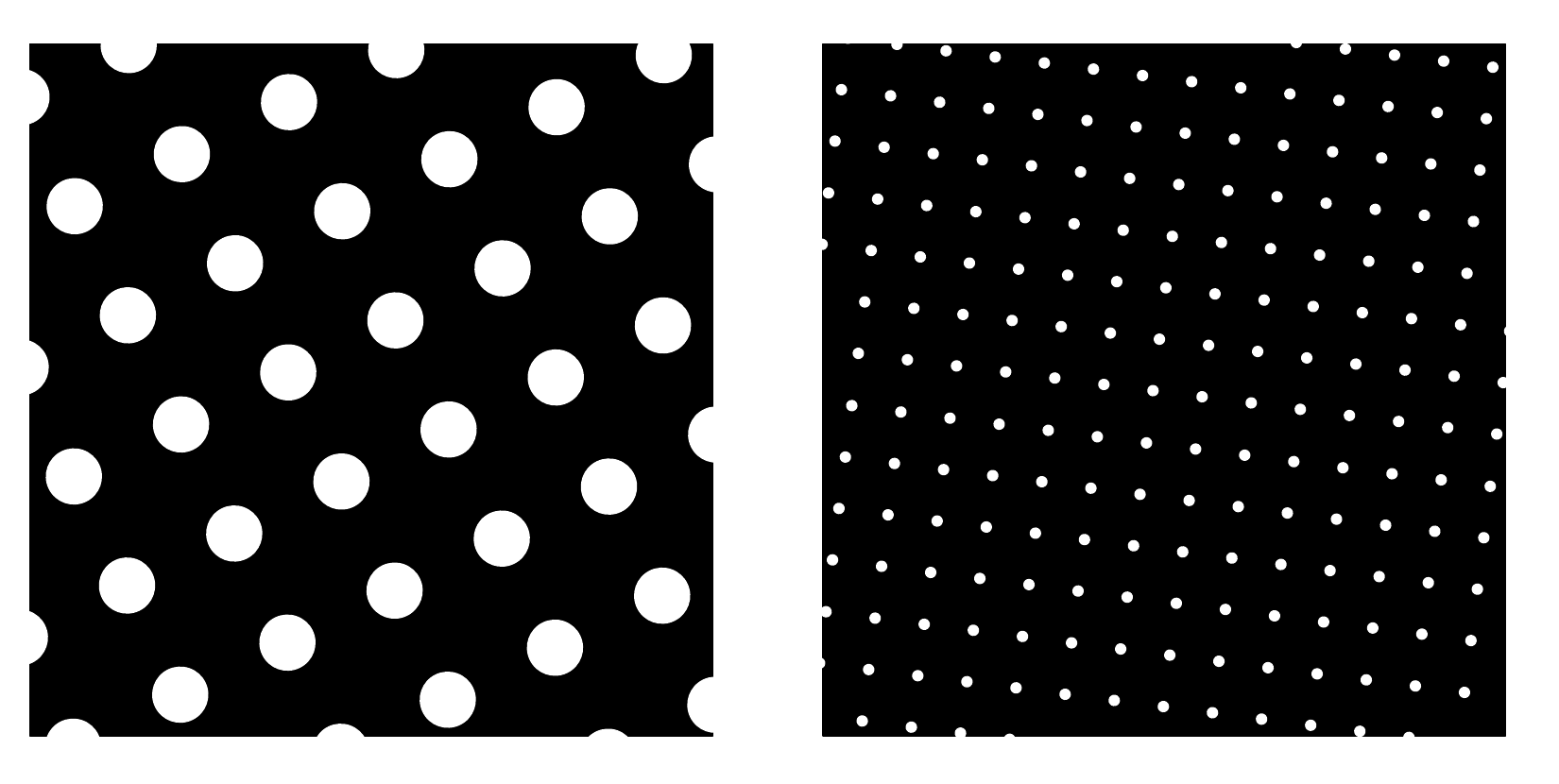}
    \caption{For $X$ in Equation~\eqref{eqn:counter}, we show the excursion set $E_X(1)$ (in black) for moderate $E > 1$ (left panel), and large $E > 1$ (right panel). }
    \label{fig:counter_example_polka_dot}
\end{figure}
 
Let $E\sim \mathrm{Exp}(1)$, $\theta\in \mathrm{Unif}([0,2\pi])$, and $U \sim \mathrm{Unif}([0,1]^2)$ be independent, and consider the stationary, isotropic  {  $2-$dimensional  } random field $\{X(s)\}_{s\in\Reals^2}$ defined by
\begin{equation}\label{eqn:counter}
    X(s) = E\left(1 - \sum_{q\in\mathbb{Z}^2} \I{\norm{q+U-E R_\theta(s)} < 3^{-E}}\right),
\end{equation}
where $R_\theta(s)$ is the image of $s$ rotated by an angle $\theta$ about the origin.
Notice that $X$ satisfies Assumption~\ref{assumtion}, for if the excursion set is not empty, it is the complement of the union of disks of radius $3^{-E}$ with centers on a randomly oriented square grid with spacing $E^{-1}$ (see Figure~\ref{fig:counter_example_polka_dot}).
If for some $p\in (0,1)$ we have $X(0) > u_p$, then $E > u_p$ and $X(R_{-\theta}(U)/E) = 0 \leq u_p$. Thus, $\tilde R^{(u_p)}(0) < \norm{U}/E < \sqrt{2}/u_p$ which tends to 0 as $p\to 1$. Thus, $R_0^{(u_p)}\xrightarrow[p\to1]{\P} 0$. It is not hard to check that
$$\frac{C_2^*(E_X(u_{p}))}{C_1^*(E_X(u_{p}))}\xrightarrow[p\to1]{}\infty,$$
contrary to the behaviour of $R_0^{(u_p)}$.

Intuitively, this random field is unlike many random fields that are studied in practice, since, when there is a large threshold exceedance, the excursion set covers most of the domain. It satisfies $\chi(s_1,s_2) = 1$, for all $s_1,s_2\in \Reals^2$. However, there are small dense holes in the excursion set that limit the size of the extremal range, and the density of the holes is bounded below by the height of the threshold. The combination of these behaviours is exceptionally malicious as most other random fields exhibit small,  distantly spaced exceedances over high thresholds.

\subsection{Regularly varying random fields}\label{sec:appendix_rv}

The process $X\ |_{T}$ is said to be \textit{regularly varying} with exponent $\alpha>0$ and spectral measure $\sigma$ on  $\mathcal{S}$ if there exists a function $a:\Reals^+\to\Reals^+$ such that $a(u)\to\infty$ and
$$u\P\bigg(\frac 1{\ \|X\|_T} X|_{T} \in A,\ \|X\|_T > ra(u)\bigg) \longrightarrow r^{-\alpha}\sigma(A)$$
as $u\to \infty$, for all $r > 0$ and $A\in\mathcal{B}(\mathcal{S})$ satisfying $\sigma(\partial A) = 0$ (see 
Definition~1 in \cite{dombry2015}). This means that the norm $\|X\|_T$ of the random vector, and its projection $\frac {X|_{T} }{\ \|X\|_T}$ onto the unit sphere of the norm, become stochastically independent  as the norm increases to infinity. This property is characteristic for $\ell$-Pareto processes.

\begin{defi}[Definition~4 in \cite{dombry2015}]\label{def:l-pareto}
    The random field $W:\Omega\times T \to \Reals$ is an \textit{$\ell$-Pareto random field} with exponent $\alpha\in \Reals^+$ and spectral measure $\sigma_\ell$ if
    \begin{itemize}
        \item $\P(W\in \mathcal{C}_0) = 1$,
        \item $\P(\ell(W) > u) = u^{-\alpha}$, for all $u > 1$,
        \item $W/\ell(W)$ and $\ell(W)$ are independent, and
        \item $\sigma_\ell(A) = \P(W/\ell(W) \in A)$ for $A\in \mathcal{B}(\mathcal{C}_0)$.
    \end{itemize}
\end{defi}

The measure $\sigma$ is defined with respect to the sup-norm $\|x\|_T$ and not necessarily a probability measure. When $\ell$ is different from the sup-norm, we have to operate a change of measure. In Section~\ref{sec:model_nondegen_rv}, the random fields $Y_0$ and $Y_T$ are $\ell$-Pareto with exponent $\alpha$ and respective spectral measures
$$\sigma_0(A) := \frac{1}{c}\int_{\mathcal{S}}x(0)^\alpha\I{x/x(0)\in A}\sigma(\mathrm{d}x),$$ 
and
$$\sigma_T(A) := \sigma(\mathcal{S}\cap A)/\sigma(\mathcal{S}),$$
for $A\in\mathcal{B}(\mathcal{C}_0)$, where $c := \int_{\mathcal{S}}x(0)^\alpha\sigma(\mathrm{d}x)$.
  


The link between regularly varying random fields and $\ell$-Pareto random fields is made by the following lemma.

\begin{lemm}[Theorem~3 in \cite{dombry2015}]\label{lem:rvlp}
    Let $X |_T$ be a regularly varying random field with exponent $\alpha > 0$ and spectral measure $\sigma$ on $\mathcal{S}$. Let $\ell:\mathcal{C}_0\to[0,\infty)$ be a homogeneous cost functional that is continuous at the origin, and is nonzero on a subset of $\mathcal{S}$ with positive $\sigma$ measure. Let $W$ be an $\ell$-Pareto random field associated to the cost functional $\ell$, having exponent $\alpha$ and spectral measure
    $$\sigma_{\ell}(A) := \frac{1}{c}\int_{\mathcal{S}}\ell(x)^\alpha\I{x/\ell(x)\in A}\sigma(\mathrm{d}x),\qquad A\in\mathcal{B}(\mathcal{C}_0),$$
    where $c := \int_{\mathcal{S}}\ell(x)^\alpha\sigma(\mathrm{d}x)$.
    Then,
    $$\P(u^{-1}X |_T \in A   \,|\, \ell(X |_T) > u)\xrightarrow[u\to\infty]{} \P(W\in A),\qquad A\in\mathcal{B}(\mathcal{C}_0).$$
\end{lemm}

\renewcommand\thesection{B} 

\section{Proofs}\label{ProofsProvidedResults}
\begin{proof}[Proof of Proposition~\ref{prp:cdf}]
Since $R^{(u)}_0$ is almost surely non-negative, it suffices to check that \eqref{eqn:R_distn_func} holds for $r\geq 0$. Note that by the continuity of $X$, the excursion set $E_X(u)$ is open, and so the events $\{\tilde R^{(u)}_0 > r\}$ and $\{B(0,r) \subset E_X(u)\}$ are equal. Furthermore, 
\begin{align*}
\P\big(R^{(u)}_0 > r\big) &= \P\big(B(0,r)\subset E_X(u)  \vert X(0) > u(0)\big)=\frac{\lebesgue_d(T)\P\big(B(0,r)\subset E_X(u)\big)}{\lebesgue_d(T)\P\big(X(0) > u(0)\big)}\\
&= \frac{\E\Big[\int_T \I{B(s,r)\subset E_X(u)} \d s\Big]}{\E\Big[\int_T \I{X(s) > u(s)}\d s \Big]}=\frac{\E\big[\lebesgue_d\big(\{s\in T:B(s,r)\subset E_X(u)\}\big)\big]}{\E\big[\lebesgue_d\big(E_X(u)\cap T\big)\big]}\\
&= \frac{\E\big[\lebesgue_d\big(E_X(u)_{-r}\cap T\big)\big]}{\E\big[\lebesgue_d\big(E_X(u) \cap T\big)\big]}.
\end{align*}
\end{proof}

\begin{proof}[Proof of Lemma~\ref{lem:continuity}]
Let $r > 0$ and let $T\subset\Reals^{d}$ be compact with positive Lebesgue measure. Define $A:= \{s\in\Reals^{d}: \dist(s, E_X(u)^c) = r\}$, and let $x\in A$. For each $\epsilon \in (0,r)$, the {Euclidean}  ball  $B(x,\epsilon)$ contains an open ball $\tilde B$ of radius $\epsilon/2$ such that for all $y\in\tilde B$, one has $\dist(y,E_X(u)^c) < r$. In particular, $\tilde B \cap A = \emptyset$, and so the Lebesgue density of $A$ at $x$ cannot exceed $1-2^{-d}$. By the Lebesgue differentiation theorem, the Lebesgue density of $A$ at $s$ is 1 for almost every $s\in A$. Since there are no elements of $A$ for which this holds, $\lebesgue_d(A) = 0$.

Suppose the statement of   Lemma \ref{lem:continuity} is false. Then, there exists  $r>0$ such that $\P(\tilde R^{(u)}(0) = r \, |\,  X(0) > u) > 0$, or equivalently, $\P(\tilde R^{(u)}(0) = r) > 0$. By stationarity and Fubini-Tonelli's theorem,
\begin{align*}
    0 < \P(\tilde R^{(u)}(0) = r) &= \E\bigg[\int_{T}\I{\tilde R^{(u)}(s) = r}\ \d s\bigg] = \E\big[\lebesgue_d(A)\big].
\end{align*}
Hence, we have  
 a contradiction, as $\P(\lebesgue_d(A) > 0) = 0$.
\end{proof}
 
\begin{proof}[Proof of Theorem~\ref{thm:distribution_small_r}]
    For $r > 0$,
    \begin{align*}
        \P(R_0^{(u)} \leq r) & = \frac{\P\left(\tilde R_0^{(u)} \leq r \cap X(0) > u(0)\right)}{\P\left(X(0) > u(0)\right)} \\
        & = \frac{\P\left(0 \in E_X(u) \setminus E_X(u)_{-r}\right)}{\P\left(X(0) > u(0)\right)} \\
        & = \frac{\P\left(0 \in E_X(u) \cap (E_X(u)^c \oplus B(0,r))\right)}{\P\left(X(0) > u(0)\right)} \\
        & = \frac{\P\left(0 \in (E_X(u)^c)_{r} \setminus E_X(u)^c\right)}{\P\left(X(0) > u(0)\right)}.
    \end{align*}
    Fix $n\in\N$ and let $T = [-n,n]^d \subset \Reals^d$. Let $t\in \Reals^d$ be a random element uniformly distributed on $T$. Then, by  the stationarity  of the excursion set $E_X(u)$, we may shift our reference point from the origin to $t$ and write
    \begin{align*}
        \P(R_0^{(u)} \leq r) & = \P(R_t^{(u)} \leq r) = \frac{\P\left(t \in T \cap (E_X(u)^c)_{r} \setminus E_X(u)^c\right)}{\P\left(X(0) > u(0)\right)}.
    \end{align*}
    Now, remark that the sets $T \cap (E_X(u)^c)_{r} \setminus E_X(u)^c$ and $(E_X(u)^c\cap T)_{r} \setminus (E_X(u)^c \cap T)$ are both contained in $T_r$ and are equal when intersected with $T_{-r}$ almost surely. Thus, the absolute difference of their Lebesgue measures satisfies
    \begin{equation}
        \Big| \lebesgue_d\left((E_X(u)^c\cap T)_{r} \setminus (E_X(u)^c \cap T)\right) - \lebesgue_d\left(T \cap (E_X(u)^c)_{r} \setminus E_X(u)^c\right)\Big| \leq \lebesgue_d(T_r \setminus T_{-r}),
    \end{equation}
    almost surely. Therefore,
    \begin{equation}\label{eqn:before_division_by_r}
        \bigg| \P(R_0^{(u)} \leq r) - \frac{\E\left[\lebesgue_d\left((E_X(u)^c\cap T)_{r} \setminus (E_X(u)^c \cap T)\right)\right]}{\lebesgue_d(T)\P\left(X(0) > u(0)\right)} \bigg| \leq \frac{\lebesgue_d(T_r \setminus T_{-r})}{\lebesgue_d(T)\P\left(X(0) > u(0)\right)}.
    \end{equation}
    By Assumption~\ref{assumtion}, $E_X(u)^c \cap T$
    is almost surely a positive reach set, and so we may apply the Tube formula (see Formula (3.5.2) in 
 \cite{AT11}) to write
    \begin{align*}
        \lebesgue\left((E_X(u)^c\cap T)_{r} \setminus (E_X(u)^c \cap T)\right) &= \lebesgue\left((E_X(u)^c\cap T)_{r}\right) - \lebesgue\left(E_X(u)^c \cap T\right) \\
        &= \lebesgue_{d-1}\left(\partial(E_X(u)^c \cap T)\right) \times r + \mathcal O (r^2),
    \end{align*}
    where equality holds almost surely. A division of Equation~\eqref{eqn:before_division_by_r} by $r$ yields
    \begin{equation*}
        \bigg| \frac{\P(R_0^{(u)} \leq r)}{r} - \frac{\E\left[\lebesgue_{d-1}\left(\partial(E_X(u)^c \cap T)\right)\right]}{\lebesgue_d(T)\P\left(X(0) > u(0)\right)} + \mathcal O(r) \bigg| \leq \frac{\lebesgue_d(T_r \setminus T_{-r})}{r\lebesgue_d(T)\P\left(X(0) > u(0)\right)},
    \end{equation*}
    and
    \begin{equation*}
        \limsup_{r\to 0} \bigg| \frac{\P(R_0^{(u)} \leq r)}{r} - \frac{\E\left[\lebesgue_{d-1}\left(\partial(E_X(u)^c \cap T)\right)\right]}{\lebesgue_d(T)\P\left(X(0) > u(0)\right)} + \mathcal O(r) \bigg| \leq \frac{2d}{n\P\left(X(0) > u(0)\right)},
    \end{equation*}
    where $n$ is half the side length of the hypercube $T$. Since $n$ can be taken arbitrarily large, we have the desired result 
    \begin{equation*}
        \lim_{r\to 0} \frac{\P(R_0^{(u)} \leq r)}{r} = \lim_{n\to\infty} \frac{\E\left[\lebesgue_{d-1}\left(\partial(E_X(u)^c \cap T)\right)\right]}{\lebesgue_d(T)\P\left(X(0) > u(0)\right)} = \frac{2C^*_{d-1}(E_X(u)^c)}{C^*_{d}(E_X(u))} = \frac{2C^*_{d-1}(E_X(u))}{C^*_{d}(E_X(u))}.
    \end{equation*}
\end{proof}

\begin{proof}[Proof of Proposition~\ref{prp:gaussian_nondegen_limiting}]

\label{sec:proofs_gaussian_nondegen_limiting}

    In \cite{kac1959}, it is shown that for a one-dimensional centered Gaussian process $\{Y(t)\}_{t\in\Reals}$ having the stationary covariance function in \eqref{eqn:gaussian_correlation}, it holds that
  $$u\big(Y(t/u)-u\big)   | \{Y(0) = u,\, Y'(0) > 0\} \xrightarrow[u\to\infty]{\mathrm{d}} -\frac{\lambda}{2}t^2 + \xi_\lambda t,$$
    where the convergence holds for the finite dimensional distributions of the process (in $t$), and $\xi_\lambda$ is some random variable that depends on $\lambda$ and the sense of the conditioning on the event $Y(0) = u$, but not on $t$.
    Therefore, $u\big(Y(t/u)-u\big)   | Y(0) > u$ converges in the same sense to $-\frac{\lambda}{2} t^2 + \xi_1t + \xi_2$, where $\xi_1$ and $\xi_2$ are random variables, and $\xi_2 > 0$ almost surely.

    Now, focusing on the $d$-dimensional random field $X$, we see that $X$ evaluated on any one-dimensional affine linear subspace of $\Reals^d$   is a Gaussian process, and so the analysis in the preceding paragraph applies to these processes. Therefore, seen as a process in~$t$,
    \begin{equation}\label{eqn:random_disk_gaussian}
    u\big(X(t/u)-u\big)   \,|\, X(0) > u \xrightarrow[u\to\infty]{\mathrm{d}} -\frac{\lambda}{2}  \|t \| ^2 + \langle\widetilde{\xi_\lambda}, t\rangle + \xi,
    \end{equation}
    where $\xrightarrow{\mathrm{d}}$
 represents the convergence  in distribution,  for some random vector $\widetilde{\xi_\lambda}$ that depends on $\lambda$ but not on $t$, and for some almost surely positive random variable $\xi$. One can show that $\xi\sim\mathrm{Exp}(1)$ independently of $\lambda$ and $t$.

    Now,
    \begin{align*}
    uR_0^{(u)} &\stackrel{\d}{=} \sup\Big\{r \in \Reals^+ : B(0,r/u) \subset E_X(u)\Big\}  | X(0) > u\\
    &\stackrel{\d}{=} \sup\Big\{r \in \Reals^+ : B(0,r) \subset uE_X(u)\Big\}   | X(0) > u\\
    &\stackrel{\d}{=} \sup\Big\{r \in \Reals^+ : B(0,r) \subset E_{X(\cdot/u)}(u)\Big\}   | X(0) > u,
    \end{align*}
    \begin{sloppypar}where $\stackrel{\d}{=}$ represents  the equality in distribution.   Equation~\eqref{eqn:random_disk_gaussian} implies that $E_{X(\cdot/u)}(u)   | X(0) > u$ converges to a non-degenerate, random hypersphere  containing the origin as $u\to \infty$, which finishes the proof.\end{sloppypar}
\end{proof}

\begin{proof}[Proof of Proposition~\ref{prp:regularly_varying_extremal_range}]
    We begin by showing the second equality in~\eqref{eqn:rv_limits}. By the continuity of $X$, the excursion set $E_X(u)$ is open, and so the events $\{\tilde R^{(u)}_0 > r\}$ and $\{B(0,r) \subset E_X(u)\}$ are equal. Also, $E_X(u) = E_{X/u}(1)$, and so for $r < r_T$,
    \begin{align}\label{eqn:rv_convergence_2nd_equality}
    \P(R_0^{(u)} > r) &= \P(B(0,r) \subset E_X(u)   | X(0) > u) = \P(B(0,r) \subset E_{X/u}(1)   | X(0) > u)\\
    &\xrightarrow[u\to\infty]{} \P(B(0,r) \subset E_{Y_0}(1)) = 1 - \P\big(\exists t\in B(0,r): Y_0(t) \leq 1\big).\nonumber
    \end{align}
    The convergence in~\eqref{eqn:rv_convergence_2nd_equality} holds by Lemma~\ref{lem:rvlp}.
    To show the first equality in~\eqref{eqn:rv_limits}, remark that $(E_{Y_T}(u) \cap T)_{-r} = E_{Y_T}(u)_{-r} \cap T_{-r}$, almost surely.  Finally, recall from Proposition~\ref{prp:cdf} that
    \begin{align*}
    \P(R_0^{(u)} > r) &= \frac{\E\big[\lebesgue_d\big(E_X(u)_{-r}\cap T_{-r}\big)\big]}{\E\big[\lebesgue_d\big(E_X(u)\cap T_{-r}\big)\big]} = \frac{\E\big[\lebesgue_d\big((E_{X/u}(1)\cap T)_{-r}\big)  \,|\,   \|X  \|_T > u\big]}{\E\big[\lebesgue_d\big(E_{X/u}(1)\cap T_{-r}\big)  \,|\,   \|X  \|_T > u\big]}\\
    &\xrightarrow[u\to\infty]{} \frac{\E\left[\lebesgue_d\big((E_{Y_T}(1) \cap T)_{-r}\big)\right]}{\E\left[\lebesgue_d\big(E_{Y_T}(1) \cap T_{-r}\big)\right]} = \frac{\E\left[\lebesgue_d\big((E_{Y_T}(1) \cap T)_{-r}\big)\right]}{\lebesgue_d(T_{-r})\P({Y_T}(0) > 1)}.
    \end{align*}
    Hence the result.
\end{proof}

\end{document}